\documentclass{article}
\usepackage{amsmath, amsfonts, amsthm, amssymb}
\usepackage{enumerate}
\usepackage{tikz}
\usetikzlibrary{matrix,arrows}
\usepackage{hyperref}

\newtheorem{theorem}{Theorem}[section]
\newtheorem{corollary}[theorem]{Corollary}
\newtheorem{definition}[theorem]{Definition}
\newtheorem{lemma}[theorem]{Lemma}
\newtheorem{proposition}[theorem]{Proposition}
\newtheorem{remark}[theorem]{Remark}

\newcommand{\bra}{\langle}
\newcommand{\ket}{\rangle}

\DeclareMathOperator{\Aut}{Aut}

\DeclareMathOperator{\ch}{char}
\DeclareMathOperator{\Ind}{Ind}

\DeclareMathOperator{\D}{\mathsf{D}}
\DeclareMathOperator{\F}{\mathbb{F}}

\DeclareMathOperator{\w}{\theta}

\begin{document}

\title{The Noether number for the groups with a cyclic subgroup of index two}             
\author{K\'alm\'an Cziszter $^a$ 
\thanks{The paper is based on results from the PhD thesis of the first author written at the Central European University.}
\qquad and \qquad M\'aty\'as Domokos $^b$ 
\thanks{The second author is partially supported by OTKA  NK81203 and K101515.}
}
\date{}
\maketitle 
{\small \begin{center} 
$^a$ Central European University, Department of Mathematics and its Applications, 
N\'ador u. 9, 1051 Budapest, Hungary \\
Email: cziszter\_kalman-sandor@ceu-budapest.edu \\
 $^b$ R\'enyi Institute of Mathematics, Hungarian Academy of Sciences,\\
 Re\'altanoda u. 13-15, 1053 Budapest, Hungary \\
Email: domokos.matyas@renyi.mta.hu
\end{center}
}

\begin{abstract} For  the finite groups with a cyclic subgroup of index two the exact degree bound for  the generators of rings of polynomial invariants  is determined. 
\end{abstract}


\section{Introduction}\label{sec:intro} 
 
The Noether number  $\beta(G)$ of a finite group $G$ is $\sup_V\beta(G,V)$, where $V$ ranges over all finite dimensional $G$-modules $V$ over a fixed base field $\F$, and 
$\beta(G,V)$ is the smallest integer $d$ such that 
the algebra $\F[V]^G := \{f \in \F[V]: f^g =f\; \forall g \in G \}$  of polynomial invariants 
is generated by its elements of degree at most $d$. By Noether's classic result \cite{noether:1916} we have 
$\beta(G) \le |G|$ if $\ch(\F)=0$, and  Fleischmann \cite{fleischmann} and   Fogarty \cite{fogarty} proved the same inequality when $\ch(\F)$ does not divide the order of $G$.  
 Recently, it has been proved that |apart from four particular groups  of small order| 
the inequality $\beta(G) \ge \frac 1 2 |G|$ holds only if $G$ is cyclic or $G$ has a
cyclic subgroup of index two (see Theorem 1.1 in \cite{CzD:1}). It is well known and easy to see that for the cyclic group $Z_n$ we have $\beta(Z_n)=n$. 
The main result of  the present article is Theorem~\ref{thm:betaindex2}, giving the precise value of $\beta(G)$ 
for every non-cyclic group containing a cyclic subgroup of index $2$. It turns out that for these groups the difference $\beta(G) - \frac 1 2 |G|$  equals $1$ or $2$. 
Despite the longstanding interest in the Noether number of finite groups, there are relatively few groups for which the exact value is known. 
It is of some interest therefore that a few infinite series of groups is added now to the list of groups with known Noether number.

A remarkable  consequence of Theorem~\ref{thm:betaindex2} and the main result of \cite{CzD:1} is that   for any constant $c>1/2$, up to isomorphism  there are only finitely many non-cyclic groups $G$ with $\beta(G)/|G|>c$, whereas there are infinitely many isomorphism classes of groups $G$ with $\beta(G)/|G|>1/2$. 
In particular,  $1/2$ is a limit point in the set $\{\beta(G)/|G| : G\mbox{ is a  finite group}\}\subset \mathbb{Q}$, and there are no other limit points  between $1/2$ and $1$.

In Section~\ref{sec:prel} we recall the generalized Noether numbers $\beta_k(G)$ and some related reduction lemmata introduced in \cite{CzD:1}. 
They play an essential role in the proof of Theorem~\ref{thm:betaindex2}. 
In Section~\ref{sec:lower} we give a general lower bound on the Noether number of a group $G$ with a normal subgroup $N$ such that $G/N$ is abelian, in terms of the Noether numbers of $N$ and $G/N$. In the following sections we investigate the generalized Noether numbers of various groups; in the last section these results will be combined with the aid of the reduction lemmata to yield a proof of the main result. 
First the appearance of zero-sum sequences in the problem is explained in Section~\ref{sec:GA}. 
The dihedral group $D_{2n}$ and some relatives are investigated in Section~\ref{sec:dihedral}. Moreover, some additional information on the indecomposable invariants of degree 
$\beta(D_{2n})$ is derived in Section~\ref{sec:extremal}, that is the basis of computation of the Noether number of certain central extensions of $D_{2n}$.  
The groups $Z_r\rtimes_{-1}Z_{4d}$ (where $r$ and $4d$ are co-prime integers, $r\ge3$) are treated in Section~\ref{sec:kontrakcio} by the so-called contraction method, that seems to be applicable in other situations. (A necessary combinatorial statement is proved in the previous Section~\ref{sec:zero-sum}.) 
The direct product of the  quaternion group of order $8$ and an odd order cyclic group needs a separate treatment performed in Section~\ref{sec:quaternion}. 
Finally, after recalling the list of the groups with a cyclic subgroup of index $2$, 
in Section~\ref{sec:index2} we combine the results in the earlier sections to derive the main result Theorem~\ref{thm:betaindex2}, giving the exact value 
of the generalized Noether numbers for each non-cyclic  group with a cyclic subgroup of index $2$.

\section{Preliminaries}\label{sec:prel}

Throughout this paper $\F$ denotes our base field, and $G$ will be a finite group with $\ch(\F)\nmid |G|$. 
By a $G$-module we mean here a finite dimensional $\F$-vector space endowed with a linear action of the finite group $G$. 
Note that $\beta(G)$ is unchanged if we replace $\F$ by its algebraic closure, therefore we shall assume that $\F$ is algebraically closed. 
Given a finitely generated graded module $M=\bigoplus_{d=0}^\infty M_d$ over a commutative graded $\F$-algebra $R=\bigoplus_{d=1}^\infty R_d$ with 
$R_0=\F$ and  $s\in\mathbb{N}$ write   $M_{\le s}:=\bigoplus_{d=0}^s M_d$ and $M_{\ge s}:=\bigoplus_{d=s}^\infty  M_d$. For a positive integer $k$ define 
 \[\beta_k(M,R):= \min\{s\in\mathbb{N}\mid M\mbox{ is generated as an }R_+^k\mbox{-module by } M_{\le s}\}\]
where $R_+^k$ is the $k$-th power of the maximal homogeneous ideal $R_+:=\bigoplus_{d=1}^\infty R_d$ of $R$. 
By the graded Nakayama Lemma $\beta_k(M,R)$ is the maximal degree of a non-zero homogeneous component of the factor space 
$M/R_+^kM$ (inheriting the grading from $M$). 
Viewing $R_+$ as an $R$-module we write 
\[\beta_k(R):=\beta_k(R_+,R).\] 

The \emph{generalized Noether numbers} of the finite group $G$ were introduced in \cite{CzD:1} as follows: 
for a $G$-module $V$ write 
\[\beta_k(G,V):=\beta_k(\F[V]^G)\]  where $\F[V]$ is the symmetric tensor algebra of $V^*$, the dual of $V$, so 
$\F[V]$ is a $\dim(V)$-variable polynomial ring endowed with its standard grading. In particular, $\F[V]_1=V^*$.  
Moreover,  set 
\[\beta_k(G):=\sup_V\beta_k(G,V)\]
where $V$ ranges over all $G$-modules over $\F$. In the special case $k=1$ we recover the  Noether number. 
The finiteness of $\beta_k(G)$ follows from the obvious inequality $\beta_k(G)\leq k\beta(G)$; this inequality is strict in general (see \cite{CzD:3} for more information in this respect), and the usefulness of the concept of the  generalized Noether numbers stems from the following statements proved in \cite{CzD:1} 
(below for subsets $S,T$ in a commutative $\F$-algebra we write $ST$ for the $\F$-vector space spanned by $\{st\mid s\in S,t\in T\}$, and 
$S^d:=S\dots S$ (with $d$ factors): 

\begin{lemma}\label{lemma:reduction} 
Let $H$ be a subgroup of $G$ and $V$ a $G$-module. 
\begin{itemize} 
\item[(i)] We have $(\F[V]_+^H)^{[G:H]}\subseteq \F[V]_+^H\F[V]_+^G+\F[V]_+^G$. 
\item[(ii)] We have $\beta_k(\F[V]_+,\F[V]^G)\le \beta_{k[G:H]}(\F[V]_+,\F[V]^H)$. 
In particular, 
\[\beta_k(G,V)\le \beta_{k[G:H]}(H,V).\] 
\item[(iii)] If $H$ is normal in $G$, then $\beta_k(G,V)\le \beta_{\beta_k(G/H)}(H,V)$. 
\end{itemize}
\end{lemma}

For later use we recall the \emph{relative transfer map} 
\[ \tau_H^G(u) = \sum_{i=1}^n u^{g_i} \] 
where $g_1,\dots,g_n$ is a system of right $H$-coset representatives in $G$. 
(In the special case when $H$ is the $1$-element subgoup of $G$ we write $\tau^G$ instead of $\tau^G_{\{1\}}$.)  
The map $\tau^G_H$  is a graded $\F[V]^G$-module epimorphism from $\F[V]^H$ onto $\F[V]^G$. 
We shall use this fact most frequently in the following form: 

\begin{proposition} \label{trans}
We have 
$\beta_k(G,V)\le \beta_k(\F[V]^H_+, \F[V]^G)$. 
\end{proposition}

It was shown in \cite{schmid} that for an abelian group $A$ we have $\beta(A)=\D(A)$, the Davenport constant of $A$, defined as the maximal length of an irreducible zero-sum sequence over $A$. (For definitions and notation related to zero-sum sequences see Section~\ref{sec:GA}.) 
The generalized Noether number also has its ancestor for abelian groups, namely 
$\beta_k(A)=\D_k(A)$, the $k$th generalized Davenport constant of $A$ introduced in \cite{halter-koch} as the maximal length of a zero-sum sequence over $A$ that does not factor as the product of $k+1$ non-empty zero-sum sequences over $A$.

\section{A lower bound} \label{sec:lower}

Schmid \cite{schmid} proved that the Noether number is monotone with respect to taking  subgroups. This extends for the generalized Noether number as well: 

\begin{lemma}\label{lemma:monotone} 
Let $W$ be a finite dimensional $H$-module, where $H$ is a subgroup of a finite group $G$, 
and denote by $V$ the $G$-module induced from $W$. 
Then the inequality $\beta_k(G,V)\geq \beta_k(H,W)$ holds for all positive integers $k$. 
\end{lemma}

\begin{proof} View $W$ as an $H$-submodule of 
\begin{equation}\label{eq:inducedrep}V=\bigoplus_{g\in G/H}gW\end{equation}  
where $G/H$ stands for a system of left $H$-coset representatives. 
Restriction of functions from $V$ to $W$ is  a graded $\F$-algebra surjection  
$\phi:\F[V]\to\F[W]$.  Clearly  $\phi$ is $H$-equivariant, hence maps $\F[V]^G$ into $\F[W]^H$. Even more, as observed in the proof of Proposition 5.1 of \cite{schmid},  we have $\phi(\F[V]^G)=\F[W]^H$: indeed, 
the projection from $V$ to $W$ corresponding to the direct sum decomposition \eqref{eq:inducedrep} identifies 
$\F[W]$ with a subalgebra of $\F[V]$, and for an arbitrary $f\in \F[W]^H\subset\F[W]\subset \F[V]$, we get that 
$\tau(f):= \sum_{g^{-1}\in G/H}  f^g\in \F[V]^G$ is a $G$-invariant mapped to $f$ by $\phi$. 
Hence if $\F[V]^G_d\subseteq (\F[V]^G_+)^{k+1}$ for some  integer $d>0$ then 
$\F[W]^H_d=\phi(\F[V]^G_d)\subseteq \phi((\F[V]^G_+)^{k+1})=(\F[W]^H_+)^{k+1}$.  
By definition of the generalized Noether number we get that $\beta_k(G,V)\geq \beta_k(H,W)$. 
\end{proof} 

\begin{corollary}\label{cor:monotone} Let $H$ be a subgroup of  a finite group $G$. Then for all positive integers $k$ we have the inequality 
$\beta_k(H)\leq \beta_k(G)$. 
\end{corollary}

Next we give a strengthening of Corollary~\ref{cor:monotone} in the special case when $H$ is normal in $G$ and the factor group $G/H$ is abelian. 
For a character $\theta\in \widehat{G/H}$ denote by $\F[V]^{G,\theta}$ the space 
$\{f\in\F[V]\mid f^g=\theta(g) f\quad \forall g\in G\}$ of the \emph{relative $G$-invariants of weight} $\theta$. Generalizing the construction in the  proof of Lemma~\ref{lemma:monotone}, for 
$f\in \F[W]^H\subset \F[V]$ (here again $V = \Ind_H^G W$) set 
\[\tau^{\theta}(f):=\sum_{g^{-1}\in G/H} \theta(g)^{-1} f^{g} \in \F[V]^{G,\theta}.\] 
Then $\phi(\tau^{\theta}(f))=f$, hence 
\begin{equation}\label{eq:fifvgteta}
\phi(\F[V]^{G,\theta})=\F[W]^H \text{ holds for all }\theta\in\widehat{G/H}.
\end{equation} 

Let $U:=\bigoplus_{i=1}^dU_i$ be a direct sum of one-dimensional $G/H$-modules $U_i$. 
Making the identification
$\F[U\oplus V]=\F[U]\otimes \F[V]
=\bigoplus_{\alpha\in \mathbb{N}_0^d}x^{\alpha}\otimes \F[V]$, 
where the variables $x_1,\ldots,x_d$ in $\F[U]$ are $G/H$-eigenvectors with weight denoted by $\theta(x_i)$ (see  the conventions  in Section~\ref{sec:GA}), we have
\begin{equation}\label{eq:fuplvg}
\F[U\oplus V]^G=\bigoplus_{\alpha\in\mathbb{N}_0^d}x^{\alpha}\otimes \F[V]^{G,-\theta(x^{\alpha}) }
\end{equation}
Setting $\tilde\phi:=\mathrm{id}\otimes \phi:\F[U\oplus V]\to\F[U]\otimes \F[W]$, 
\eqref{eq:fifvgteta} and \eqref{eq:fuplvg} imply  that
\begin{equation}\label{eq:idotimesphi}
 \tilde\phi(\F[U\oplus V]^G_+)=\F[U]^G_+\oplus \bigoplus_{\alpha\in\mathbb{N}_0^d}x^{\alpha}\otimes 
 \F[W]^H_+. 
\end{equation}

\begin{theorem}\label{thm:lower} 
Let $H$ be a normal subgroup of a finite group $G$ with $G/H$ abelian. Then for all positive integers $k$ we have the inequality 
\[\beta_k(G)\geq \beta_k(H)+\D(G/H)-1.\] 
\end{theorem} 

\begin{proof} Take $W,V,U=\bigoplus_{i=1}^dU_i$ as above, where  we have $\beta_k(H)=\beta_k(H,W)$ in addition, 
and the characters $\theta_1,\ldots,\theta_d$ of the summands $U_i$  constitute  a maximal length zero-sum free sequence over the abelian group $\widehat{G/H}$ 
(see Section~\ref{sec:GA} for zero-sum sequences). 
In particular, $d=\D(G/H)-1$ (since  $\F$ is assumed to be algebraically closed). 
Choose a homogeneous $H$-invariant  $f\in \F[W]^H$  of degree $\beta_k(H,W)$, 
not contained in $(\F[W]^H_+)^{k+1}$, and consider the $G$-invariant 
\[t:=x_1\cdots x_d\otimes \tau^{\theta}(f)\in \F[U\oplus V]^G,\] 
where 
$\theta=\sum_{i=1}^d\theta_i$ (we write the character group $\widehat{G/H}$ additively). 
Then $t\in \F[U\oplus V]^G$ is homogeneous of degree $d+\beta_k(H,W)$.  
We will show that $t\notin (\F[U\oplus V]_+^G)^{k+1}$, implying 
$\beta_k(G,U\oplus V)\geq  \beta_k(H,W)+d=\beta_k(H)+d$. 
Indeed, assume to the contrary that $t\in  (\F[U\oplus V]_+^G)^{k+1}$. Then by \eqref{eq:idotimesphi} we have 
\[x_1\cdots x_d\otimes f=\tilde\phi(t)\in \left(\F[U]^G_+\oplus \bigoplus_{\alpha\in\mathbb{N}_0^d}x^{\alpha}\otimes \F[W]^H_+\right)^{k+1}.\] 
Since $\F[U]^G_+$ is spanned by monomials not dividing the monomial $x_1\cdots x_d$ (recall that $\theta_1,\ldots,\theta_d$ is a zero-sum free sequence), we conclude that 
\begin{equation}\label{eq:xotimesf}
x_1\cdots x_d\otimes f\in \left(\bigoplus_{\alpha\in\mathbb{N}_0^d}x^{\alpha}\otimes \F[W]^H_+\right)^{k+1}.
\end{equation}
Denote by $\rho:\F[U]\otimes \F[V]\to\F[V]$ the $\F$-algebra homomorphism given by the specialization $x_i\mapsto 1$ ($i=1,\ldots,d$). 
Applying $\rho$ to (\ref{eq:xotimesf}) we get that 
$f\in (\F[W]^H_+)^{k+1}$, contradicting  the choice of $f$. 
\end{proof} 

\begin{remark} 
(i) The proof of  Theorem~\ref{thm:lower} also yields the stronger conclusion 
\begin{align} 
\beta_k(G) \ge \max_{0 \le s \le k-1}  \beta_{k-s}(H) + \D_{s+1}(G/H) -1 
\end{align}

(ii) If $G$ is abelian, we get  
$\D_k(G)\geq \D_k(H)+\D(G/H)-1$  for any subgroup $H\le G$. 
For the case $G=H\oplus H_1$, this was proved in \cite{halter-koch}, Proposition~3~(i). 
\end{remark}


\section{The role of zero-sum sequences}\label{sec:GA} 

In the rest of the paper we shall deal with the following situation: there is a distinguished non-trivial abelian normal subgroup $A$ in $G$, 
and any $G$-module $V$ has an $A$-eigenbasis permuted up to non-zero scalar multiples by $G$. 
This holds for example when $A$ is an index two subgroup, since then 
an irreducible $G$-module is either $1$-dimensional or is induced from a $1$-dimensional $A$-module. 
We shall always tacitly assume that our variables $x_1,\dots,x_n$ are permuted up to non-zero scalar multiples by $G$ and 
$x_i^a=\theta_i(a)x_i$ for all $a\in A$, where $\theta_i:A\to \F^\times$ is a character of $A$, called the \emph{weight}  of $x_i$.  
The set of characters of $A$ is denoted by $\hat{A}$;  there is a (non-canonic) isomorphism $\hat{A} \cong A$ of abelian group, and we shall write 
$\hat A$ additively. 
Let $M(V)$ denote the set of monomials in $\F[V]$; this is a monoid with respect to ordinary multiplication and unit element $1$. 
On the other hand we denote by $\mathcal{M}(\hat{A})$ the free commutative monoid generated by the elements of $\hat{A}$.
Define a monoid homomorphism $\Phi: M(V) \to \mathcal{M}(\hat{A})$
by sending each variable $x_i$ to its weight $\theta_i$.  
We shall call $\Phi(m)$ the \emph{weight sequence} of the monomial $m \in M(V)$. 

An element $S \in \mathcal{M}(\hat{A})$ can  be interpreted as a \emph{sequence} 
$S:=(s_1,\ldots,s_n)$ of elements of $\hat{A}$ where their order is disregarded and repetition of elements is allowed; we call the number occurrences of an element its 
\emph{multiplicity} in $S$. 
The \emph{length} of $S$ is $|S|:=n$. 
By a \emph{subsequence} of $S$ we mean $S_J := (s_j\mid j\in J)$ for some subset $J\subseteq \{1,\ldots,n\}$. %
Given a sequence $R$ over an abelian group $A$ we write 
$R=R_1R_2$ if $R$ is the concatenation of its subsequences $R_1$, $R_2$,
and we call the expression $R_1R_2$ a \emph{factorization} of $R$. 
Given an element $a\in A$ and a positive integer $r$, write $(a^r)$ for the sequence in which $a$ occurs with multiplicity $r$. 
For an automorphism $b$ of $A$ and a sequence $S=(s_1,\dots,s_n)$ we write $S^b$ for the sequence 
$(s_1^b,\dots,s_n^b)$, and we say that the sequences $S$ and $T$ are \emph{similar} if $T=S^b$ 
for some $b \in \Aut(A)$.

Let $\theta: \mathcal{M}(\hat{A}) \to \hat{A}$ be the monoid homomorphism 
which assigns to each sequence over $A$ the sum of its elements. 
The value $\theta(\Phi(m)) \in \hat A$ is called the \emph{weight of the monomial} $m \in M(V)$ and it will be abbreviated by $\w(m)$. 
The sequence $S$ is a \emph{zero-sum sequence} if $\theta(S)=0$. 
Our interest in zero-sum sequences 
and the related results in additive number theory  stems from the observation that
the invariant ring $\F[V]^A$ is spanned as a vector space by all those monomials
for which $\Phi(m)$ is a zero-sum sequence over $\hat A$. Moreover, as an algebra, $\F[V]^A$ is minimally 
generated by those monomials $m$ for which $\Phi(m)$ 
does not contain any proper zero-sum subsequences. 
These are  called \emph{irreducible} zero-sum sequences.  
A sequence is \emph{zero-sum free} if it has no non-empty zero-sum subsequence. 
See for example \cite{geroldinger-gao} for a survey on zero-sum sequences.


\section{Groups of dihedral type } \label{sec:dihedral}

\begin{definition}\label{def:zerocorner} 
A sequence $C$ over an abelian group $A$ is called a 
\emph{zero-corner} if $C$ has a factorization $C=EFH$ into non-empty subsequences $E,F,H$ such that $EF$ and $EH$ are zero-sum sequences. 
We denote by $\rho(C)$  
the minimal value of $\max \{ |EF|, |EH|, |FH| \}$ 
over all factorizations $C=EFH$ satisfying the above properties,
and we call it the  \emph{diameter} of $C$. 
\end{definition}

\begin{lemma}\label{lemma:zerocorner}
Let $S=(s_1,\ldots,s_l)$ be a sequence over $A$ consisting of non-zero elements. Suppose that $S$ contains a maximal zero-sum free subsequence of length $d\leq l-3$. Then $S$ contains   a zero-corner $C$ with $\rho(C) \leq d+1$.
\end{lemma}

\begin{proof}  For $I\subseteq \{1,...,l\}$ we denote by $S_I$ the subsequence $(s_i : i \in I)$. 
We may suppose that a maximal zero-sum free subsequence of $S$ is $S_J$ where $J = \{1,...,d\}$. 
For each $i = 1, 2, 3$ a nonempty subset $H_i \subseteq J \cup \{d+i\}$ exists  such that $S_{H_i}$ is an irreducible zero-sum sequence and $d+i\in H_i$.  
Observe that $|H_i| \geq  2$ as the zero-sum sequence $S_{H_i}$ must consist of non-zero elements. 
There are two cases: 
\begin{itemize}
\item[(i)] If the three sets $H_i$ are pairwise disjoint then  $C := S_{H_1} S_{H_2} S_{H_3}$ is  a zero-corner with $\rho(C) \leq d + 3 -\min\{|H_1|, |H_2|, |H_3|\} \leq d + 1$. 
\item[(ii)] Otherwise, if e.g. $H_1 \cap H_2 \neq\emptyset$ then $C := S_{H_1 \cup  H_2}$	is a zero-corner with $\rho (C) \leq \max\{|H_1|,|H_2|,d + 2 - |H_1 \cap H_2|\} \leq d + 1$; indeed, $C=EFH$ 
with $E:=S_{H_1\cap H_2}$, $F:=S_{H_1\setminus H_2}$, $H:=S_{H_2\setminus H_1}$.  \qedhere
\end{itemize} 
\end{proof}

We turn now to a semidirect product $G=A\rtimes_{-1} Z_2$ where $A$ is a  non-trivial  abelian group and $Z_2 = \bra b \ket$ acting on it by inversion (in particular, when $A=Z_n$ is the cyclic group of order $n$, we obtain the dihedral group $D_{2n}$ of order $2n$). 
Keeping conventions, notations and terminology introduced in Sections~\ref{sec:prel} and \ref{sec:GA}, 
let $W$ be a  $G$-module over $\F$,  $I=\F[W]^A$, $R=\F[W]^G$ and $\tau:=\tau^G_A:I \to R$ is the relative transfer map. 

\begin{proposition}\label{prop:dihedraltype}
For any monomial  $m\in I$ and integer $k \ge 0$ it holds that 
 $m \in I_+R_+^k$ provided that
\begin{enumerate}
\item[(i)] $\deg(m) \geq k\D(A)+2$, or 
\item[(ii)] $\deg(m) \geq (k-1) \D(A)+d+2$ where $\Phi(m)$ contains  a zero-corner with diameter $d$ 
\end{enumerate}
\end{proposition}

\begin{proof} 
We  apply induction on $k$. 
The case $k=0$ is trivial so we may suppose $k \ge 1$. 
Assume condition (ii). Thus  $m=nr$ where the monomial  $n=efh$ is such that 
$ef$ and $eh$ are $A$-invariant monomials, and  
$ \max \{ \deg(ef),\deg(eh),\deg(fh) \} =d$. 
Denoting $\w(e)$ by $a \in \hat{A}$ we have $\w(f)=\w(h)=-a$ and $\w(r)=\w(e)=a$. 
The generator $b$ of $Z_2$ transforms each monomial of weight $a$ into a monomial of weight $-a$,
and vice versa, hence $fh^b$ and $e^br$ are both $A$-invariant. 
Given that $b^2=1$ the following relation holds: 
\begin{align} \label{masodik} 2m =  
\tau(ef) hr + 
\tau(eh) fr - 
\tau(fh^b) e^b r. 
\end{align}
After division by $2 \in \F^{\times}$ we get from \eqref{masodik}  that $m \in  I_{\geq \deg(m) - d} (R_+)_{\le d}$. 
Given that $\deg(m) - d \ge (k-1)\D(A) +2$ by assumption, the induction hypothesis applies, 
whence $I_{\geq \deg(m) - d} \subseteq R_+^{k-1}I_+$ and $m \in I_+R_+^k$ as claimed. 
Suppose next that condition (i) holds. 
If $m$ contains three $A$-invariant variables, then $\Phi(m)$ contains  the zero corner $(0,0,0)$ with diameter $2$,
 hence we are back in case (ii). 
Otherwise $\Phi(m)$ contains a subsequence of length at least $k\D(A)$ of non-zero elements. 
If $k>1$, then by Lemma~\ref{lemma:zerocorner} 
$\Phi(m)$ has a zero-corner of diameter at most $\D(A)$, so again we are back in case (ii). 
It remains that $k=1$. If $m$ contains one or two $A$-invariant variables, then $m\in I_+^3\subseteq I_+R_+$ by 
Lemma~\ref{lemma:reduction}. Otherwise $m$ contains a subsequence of length at least $\D(A)+2$ of non-zero elements, 
hence by Lemma~\ref{lemma:zerocorner} $\Phi(m)$ contains a zero-corner of diameter at most $\D(A)$. We are done by case (ii).  
\end{proof}

\begin{theorem} \label{thm:dihedraltype}
Let $G =A\rtimes_{-1} Z_2$ and suppose $|G| \in \F^{\times}$.  
Then 
\[ \D_k(A) +1 \le \beta_k(G) \leq k\D(A)+1 \]
\end{theorem}

\begin{proof} By Proposition~\ref{trans} we have $\beta_k(G,W) \le \beta_k(I_+,R)$. 
Since
$I_d \subseteq I_+R_+^{k}$ for $d \ge  k\D(A)+2 $ by Proposition~\ref{prop:dihedraltype}, it follows that
$\beta_k(I_+,R) \le k\D(A)+1$.  
The lower bound is given by Theorem~\ref{thm:lower}. 
\end{proof}

If $k=1$ then $\D_1(A) = \D(A)$ and if $A=Z_n$ is cyclic then  $\D_k(Z_n) = k\D(Z_n)$, hence we obtain the following immediate consequences: 

\begin{corollary} For any abelian group $A$ we have $\beta(A \rtimes_{-1} Z_2) = \D(A) +1$.\end{corollary}

\begin{corollary}\label{cor:dihedralbeta} 
For the dihedral group $D_{2n}$ of order $2n$ and an arbitrary positive integer $k$ we have 
$\beta_k(D_{2n})=nk+1$, provided that $2n \in \F^{\times }$. 
\end{corollary} 

The special case $k=1$ of Corollary~\ref{cor:dihedralbeta} is due to Schmid \cite{schmid} when ${\mathrm{char}}(\F)=0$ and to Sezer \cite{sezer} in non-modular positive characteristic. 

\section{Extremal invariants}\label{sec:extremal} 
 
Let $A$ be an abelian normal subgroup in a finite group $G$, and assume the conditions and conventions from the beginning of Section~\ref{sec:GA}.  
\begin{definition} \label{def:extremal}
Let $R= \F[V]^{G}$;  a monomial 
$u \in \F[V]^A$ will be called  \emph{$k$-extremal with respect to $\tau_A^G$} 
if  $\deg(u)= \beta_k(G)$ while $\tau_A^G(u) \not\in R_+^{k+1}$. 
A sequence $S$ over $\hat A$ is  $k$-extremal if there is a $G$-module $V$ and a monomial $m \in \F[V]^A$ 
with  $\Phi(m) =S$ such that 
 $m$ is $k$-extremal with respect to $\tau^G_A$. 
\end{definition}

For any sequence $S = (s_1, ..., s_d)$ over an abelian group $A$ the set of its partial sums is 
\[\Sigma(S) := \{ \sum_{i\in I} s_i: I  \subseteq \{1,...,d \} \}.\]

\begin{lemma} \label{lemma:easy} 
Let $p$ be a prime and $S=(s_1,...,s_d)$ a sequence of non-zero elements of $Z_p$. 
Then $|\Sigma(S)| \ge \min \{ p, d+1 \}$. 
\end{lemma}

\begin{proof} This is a well known and easy consequence of the Cauchy-Davenport Theorem, asserting that 
$|C+D|\geq \min\{p,|C|+|D|-1\}$
for any non-empty subsets $C,D$ in $Z_p$, where $p$ is a prime. 
\end{proof}

\begin{lemma}\label{lemma:freeze-smith} (Freeze -- Smith \cite{freeze_smith}) 
For any zero-sum free sequence $S$ over $Z_n$ of length $d$ and maximal multiplicity $h=h(S)$ it holds that 
\[ |\Sigma(S)| \ge 2d-h+1.\] 
\end{lemma}

\begin{proposition} \label{extrem_Dn}
Let $G= A\rtimes_{-1} Z_2=D_{2n}$ be the dihedral group of order $2n$ where  $n\ge 3$. A sequence  over $\hat A \cong Z_n$ 
is  $k$-extremal with respect to $\tau_A^G$ only if it has the form $ (0,a^{kn})$ for some generator $a$ of $\hat A$. 
\end{proposition}
\begin{proof} 
Let $m \in \F[W]^A$ be a monomial of $\deg(m)=\beta_k(D_{2n})=kn+1$ such that $\tau_A^G(m) \not\in R_+^{k+1}$. 
If $m$ is divisible by the product of two weight zero variables, then $m\in R_+I_{\ge kn-1}$ by Lemma~\ref{lemma:reduction}. 
Since $kn-1>\beta_{k-1}(D_{2n})$, we get  $\tau^G_A(m)\in R_+\tau^G_A(I_{>\beta_{k-1}(D_{2n})})\subseteq R_+^{k+1}$, a contradiction.
It remains that the multiplicity of $0$ in $\Phi(m)$ is at most one.  
Let $H \subseteq Z_n $ be the set of nonzero values occurring in $\Phi(m)$.
Suppose  $|H| \ge 2$; if $\Phi(m)$ contains a zero-corner of the form $(w,w,-w)$  with diameter $2$, then
 $\tau(m) \in R_+^{k+1}$ by Proposition~\ref{prop:dihedraltype} (ii), a contradiction. 
 We are done if $n=3$, so assume for the rest that $n\ge 4$. 
 Then $\Phi(m)$ contains a  zero-sum free subsequence of length $2$, consisting of two distinct elements. 
 By Lemma~\ref{lemma:freeze-smith} this extends to a maximal zero-sum free subsequence of length at most $n-2$. 
If $k>1$ or $0\notin \Phi(m)$, then 
$\tau(m) \in R_+^{k+1}$ by Lemma~\ref{lemma:zerocorner} and Proposition~\ref{prop:dihedraltype}, a contradiction. 
If $k=1$ and $0\in \Phi(m)$, then $m\in I_+^3$, hence $\tau(m)\in R_+^2$ by Lemma~\ref{lemma:reduction}, 
a contradiction again. 
Consequently $|H|=1$ and  $\Phi(m)= (0, a^{kn})$. Taking into account Lemma~\ref{lemma:reduction},  
$a$ must have order $n$, whence our  claim.
\end{proof}


\section{A result on zero-sum sequences} \label{sec:zero-sum}

Let $e$ be a generator of the cyclic group $Z_n$; 
for an arbitrary element $a \in Z_n$, the smallest positive integer $r$ such that $a =re$ is denoted by $|| a ||_e$. 
For any sequence $S= (a_1,...,a_l)$ over $Z_n$ we set $||S||_e := ||a_1||_e + ... + ||a_l ||_e$.

The following two statements  are based on an intermediary step in the proof of the Savchev -- Chen Theorem (see Proposition 2. in \cite{savchev}):

\begin{proposition}\label{savchev-chen zerosumfree}
Let $S_1 \subset S_2 \subset ... \subset S_t$ be zero-sum free sequences 
over the cyclic group $Z_n$ such that $|S_i| = i$ for  all $ i=1,...,t$ and
\begin{align} \label{sav chen lemma}
|\Sigma(S_{i+1})| \ge |\Sigma(S_i)| +2 \qquad \text{ for all } i \le t-1
\end{align}
If moreover $S_t(b)$ is also zero-sum free for some $b \in Z_n$ and  $|\Sigma(S_t(b))| = |\Sigma(S_t)|+1$, 
then $b$ is the unique element with these two properties.
\end{proposition}

\begin{lemma} \label{Zn_extrem1}
Any sequence $S$ over $Z_n$   contains 
 either a zero-sum sequence of length at most $\lceil \frac{n}{2} \rceil$
or an element of multiplicity at least $|S| - \lfloor \frac{n}{2} \rfloor$.
\end{lemma}

\begin{proof}
Suppose that $S$ does not contain a zero-sum sequence of length at most $\lceil \frac{n}{2} \rceil$
and let $S_1 \subset ... \subset S_t$ be zero-sum free sequences where $t$ is maximal with the property that 
$|\Sigma(S_{i+1})| \ge |\Sigma(S_i)| +2$ and $|S_i| = i$ for every $i \le t-1$; let   $S= S_t R$. 
By this assumption $n \ge |\Sigma(S_t) | \ge 2t$. If $t= \lceil \frac{n}{2} \rceil$, which enforces that $n$ is even, 
then $|\Sigma(S_t)| = n$, hence any $a \in R$ can be completed into a zero-sum sequence $U(a)$ with some $U \subseteq S_t$. 
By our assumption it is necessary that $|U(a)| > \lceil \frac{n}{2} \rceil$, hence $U = S_t$
and the multiplicity of  $a = -\theta(S_t)$ is at least $|R| = |S| -  \frac{n}{2} $.
It remains that $t \le \lceil \frac{n}{2} \rceil -1$. 
Then for any $b \in R$ the sequence $S_t(b)$ of length at most $\lceil \frac{n}{2} \rceil$
must be zero-sum free by our assumption, hence by the maximality property of $S_t$  necessarily $|\Sigma(S(b))| = |\Sigma(S)| +1$. 
But we know from Proposition~\ref{savchev-chen zerosumfree} that the element $b$ with these two properties is unique, 
hence $b$ has  multiplicity  $|R| \ge |S| - \lceil \frac{n}{2} \rceil +1$.
\end{proof}

\begin{lemma} \label{Zn_extrem2}
Let $S$ be a zero-sum sequence over $Z_n$ of length $|S| \ge kn +1$ 
where  $k \ge 2$, 
which does not factor into more than $k+1$ non-empty zero-sum sequences. 
Then $S= T_1T_2(e^{(k-1)n})$ where $\bra e \ket = Z_n$  and $||T_1||_e = ||T_2||_e = n$.
\end{lemma}

\begin{proof}
First we prove that an element $e \in S$ has multiplicity at least $(k-1)n$;
if so $e$ will have order $n$, for otherwise $S$ factors into at least $2(k-1)+2 > k+1$
non-empty zero-sum sequences. 
Let $S=T_1S_1$  where $T_1$ is a non-empty zero-sum sequence of minimal length in $S$.
If $|T_1| > \lceil \frac{n}{2} \rceil$ then $h(S) \ge |S| - \lfloor \frac{n}{2} \rfloor$ by Lemma~\ref{Zn_extrem1}, and we are done.
If however $|T_1| \le \lceil \frac{n}{2} \rceil$ then 
 $S_1= T_2S_2$ where $T_2$ is a minimal non-empty zero-sum sequence in $S_1$;
obviously $|T_2| \ge |T_1|$. 
If $|T_2| > \lceil \frac{n}{2} \rceil$ then  
$h(S) \ge h(S_1) \ge |S_1| - \lfloor \frac{n}{2} \rfloor \ge |S| - \lceil \frac{n}{2} \rceil - \lfloor \frac{n}{2} \rfloor = |S| -n$
by Lemma~\ref{Zn_extrem1}, and we are done again.
It remains that $|T_2| \le \lceil \frac{n}{2} \rceil$. 
Then $|T_1T_2| \le n+1$ and $|S_2| \ge (k-1)n$. 
Given that $S_2$ cannot be factored into more than $k-1$ non-empty zero-sum sequences
it is necessary that $S_2 = (e^{(k-1)n})$. 
 
Now suppose to the contrary that $||T_1||_e > n$, say. Then $T_1 = U(a)V$ where $U,V$ are non-empty subsequences
such that $||U||_e < n$, $||U(a)||_e>n$.
But then  $(e^n)\cdot T_1 = (e^{n-||U||_e})U \cdot (e^{n - ||a||_e}a) \cdot (e^{||U||_e + ||a||_e-n})V $ is a  factorization 
which leads to a decomposition of $S$  into more than $k+1$ non-empty zero-sum sequences,
and this is a contradiction.
 \end{proof}

\section{The contraction method: the groups $Z_r\rtimes_{-1}Z_{2n}$} \label{sec:kontrakcio}

Let  $B\le A$ be a subgroup of an abelian group $A$. 
If $S=(s_1, ..., s_d)$ is a sequence over $A$, then $(s_1+B, ..., s_d+B)$ is a sequence over $A/B$  
which will be denoted by $S/B$. Suppose that $\theta(S) \in B$; 
a \emph{$B$-contraction} of  $S$ is a sequence over $B$ of the form $(\theta(S_1),...,\theta(S_l))$
where $S=S_1... S_l$  
and each $S_i/B$ is an irreducible zero-sum sequence over $A/B$;
so  indeed $\theta(S_i) \in B$.  

Suppose that  $A$ is a non-trivial abelian normal subgroup of $G$. 
Let $C<A$ be a subgroup such that  $C\triangleleft G$ hence $A/C$ is a non-trivial abelian normal subgroup of $G/C$. 
Suppose moreover that any $G$-module has an $A$-eigenbasis permuted up to scalars by $G$, so 
we can apply the conventions of Section~\ref{sec:GA} both for the pair $(G,A)$ and $(G/C,A/C)$. 
Note that $\widehat{A/C}$ is naturally a subgroup of $\hat A$, thus 
the  above notion of contractions can be applied based on the following observation:

\begin{lemma} \label{resolution}
For any $G$-module $V$  there exists a $G/C$-module $U$ and 
a $G/C$-equivariant $\F$-algebra epimorphism $ \pi: \F[U] \to \F[V]^C $ such that  
any monomial $m \in \F[V]^C$ has a preimage $\tilde{m} \in \pi^{-1}(m)$ with 
$\Phi(\tilde{m})$ equal to an arbitrarily prescribed $\widehat{A/C}$-contraction of $\Phi(m)$. 
\end{lemma}

\begin{proof}
By assumption $V^*$ has a basis $x_1,...,x_n$ consisting of $A$-eigenvectors which are permuted up to scalars by  $G$. 
Let $M$ be the set of $C$-invariant monomials in these variables, and $E \subset M$ the subset of the irreducibles among them,
i.e.  which cannot be factored into  two non-trivial $C$-invariant monomials. 
$\F[V]^C$ is minimally generated as an algebra by $E$. 
Moreover the factor group $G/C$ has an inherited action on ${\F}[V]^C$,  
and permutes the elements of $E$ up to non-zero scalar multiples. 
Define $U$ as the dual of the $G/C$-invariant subspace $\mathrm{Span}_{\F}(E)$. 
$E$ is a basis of this vector space, hence $E$ is identified with the set of variables in $\F[U]$. 
The $\F$-algebra epimorphism $\pi:\F[U]\to\F[V]^C$ taking a variable to the corresponding irreducible $C$-invariant monomial is $G/C$-equivariant.
Now let $(\theta(S_1),...,\theta(S_l))$ be an arbitrary $\widehat{A/C}$-contraction of $\Phi(m)$ for a monomial $m \in \F[V]^C$. 
By definition this means that $m = m_1...m_l$ where each $m_i$ is an irreducible $C$-invariant monomial with $\Phi(m_i) = S_i$. 
Hence for each $i$ there are variables $y_1,...,y_l \in \F[U]$ such that $\pi(y_i) = m_i$ by construction, 
and the monomial $\tilde{m} := y_1...y_l$ has the required property.
\end{proof}

Using this map $\pi$ we can derive  information on the generators of $\F[V]^G$
from our preexisting knowledge about the generators of $\F[U]^{G/C}$. 
As an example of this principle, we will study here the group
$G := Z_r \rtimes_{-1} Z_{2n}$
where $r$ and $2n$ are  coprime, $r \ge 3$, and the generator of $Z_{2n}$ operates by inversion on $Z_r$. 
The center of $G$ is $C = Z_{n} $ and $G/C$ is isomorphic to the dihedral group $D_{2r}$
whose extremal monomials were described before.
$G$ has the abelian normal subgroup $A \cong Z_{rn}$, $A \ge C$ such that $G/A = \bra b \ket \cong Z_2 $. 
We will write  $S\sim S'$ for two sequences over $\hat{A}$ if $S=EF$ and $S'=E^bF$ for a zero-sum sequence $E$ of length at most $n$.

\begin{proposition} \label{prop:kontra}
If $S$ is a $k$-extremal sequence over $\hat{A}$ then
any $\widehat{A/C}$-contraction of any sequence $S' \sim S$ is a $k$-extremal sequence over $\widehat{A/C}$. 
\end{proposition}

\begin{proof}
Since $S$ is a $k$-extremal sequence, there is a $G$-module $V$ and a monomial $m\in \F[V]^A$ such that $\Phi(m) =S$
and $m$ is $k$-extremal with respect to $\tau_A^G$.
Let $\pi: \F[U]^{A/C} \to \F[V]^A$ denote the restriction of the map constructed in Lemma~\ref{resolution}
to the $A$-invariants, 
and consider the transfer maps $\tilde{\tau}: \F[U]^{A/C} \to \F[U]^{G/C}$, $\tau: \F[V]^A \to \F[V]^G$. 
The $G/C$-equivariance of $\pi$ implies that $\tau \pi = \pi \tilde{\tau}$.  
Suppose first that $S$ has a non-$k$-extremal $C$-contraction $\tilde{S}$. 
Since $|\tilde S|\ge \frac{1}{n}|S|$ where we have $|S| = \beta_k(G) \ge knr +1$ by Theorem~\ref{thm:lower},
it follows that $|\tilde{S}| \ge kr+1 = \beta_k(G/C)$ by Corollary~\ref{cor:dihedralbeta}. 
So for the monomial $\tilde{m} \in \F[U]$ with $\pi(\tilde{m}) = m$ and $\Phi(\tilde{m}) = \tilde{S}$, which exists  by Lemma~\ref{resolution},
we have $\tilde{\tau}(\tilde{m}) \in (\F[U]^{G/C}_+)^{k+1}$. But then $\tau(m) = \pi(\tilde{\tau}(\tilde{m})) \in (\F[V]^G_+)^{k+1}$,
a contradiction.

Now suppose that a sequence $S' =E^bF$ has a $C$-contraction $\tilde{S}$ which is not $k$-extremal, where $0 <|E| \le n$.
Then take a factorization $m=uv$ with $\Phi(u) =E$ and $\Phi(v) = F$. 
By the previous argument  $\tau(u^bv) \in (\F[V]^G_+)^{k+1}$. 
By Lemma~\ref{lemma:reduction} and Corollary~\ref{cor:dihedralbeta} we have $\beta_{k}(G) \le \beta_{\beta_{k}(D_{2r})}(C) = nrk+n$,
hence $\deg(v)=\deg(m)-|E|\ge \beta_k(G)-n\ge nrk+1-n> nr(k-1) +n \ge \beta_{k-1}(G)$. 
Consequently  $\tau(v)\in (\F[V]_+^G)^k$ and
 $\tau(m) = \tau(u)\tau(v) - \tau(u^b v) \in (\F[V]^G_+)^{k+1}$,
a contradiction again.
\end{proof}

In the following statement we identify $\hat A=Z_{rn}$ with the additive group of $\mathbb{Z}/rn\mathbb{Z}$ and write $0,1,2,\dots$ for its elements, whenever it seems convenient. 

\begin{lemma} \label{lemma:kontra}
Let $S$ be a zero-sum sequence over $\hat{A} = Z_{rn}$ having length at least $nrk+1$, where   $k \ge 1$,  
$n\ge 3$, $r \ge 3$, and $r$ and $2n$ are coprime. 
If any $Z_r$-contraction of any sequence $S' \sim S$ is similar to $(0,n^{rk})$ then  
$S$ is similar to $(0,1^{nrk})$. 
\end{lemma}

\begin{proof}
By assumption any $Z_r$-contraction of $S$ must have length $l:=rk+1$.  
By Lemma~\ref{Zn_extrem2} then $S= T_1...T_l$ where $T_i/Z_r=(e^n)$
for every $i \le l-2$ and some generator $e$ of $Z_{rn}/Z_r \cong Z_n$, 
while $||T_{l-1}/Z_r||_e = ||T_l/Z_r ||_e = n$, and we may assume that the sequence $(\theta(T_1),\dots,\theta(T_l))$ equals $(0,n^{rk})$. 
In particular, at most one element of the sequence $S$ belongs to $Z_n$, and so $x^b\neq x$ for $x\in S$ with at most one 
exception. 
As $l \ge 4$ we may assume that $\theta(T_1) \neq 0$ and let $i \neq 1$ be any other index for which $\theta(T_i) \neq 0$.
Take an arbitrary element $x \in T_i$ and let $U \subseteq T_1$ be an arbitrary subsequence of length $d:= ||x+ Z_r||_e < n$.
After exchanging the proper subsequences $U$ and $(x)$ in $T_1$ and $T_i$
 the resulting  $\tilde{T}_1 $ and $\tilde{T}_i$  projects to zero-sum sequences over $Z_n$, so we get another $Z_r$-contraction of $S$:
\[(\theta(\tilde{T}_1), \theta(T_2),..., \theta(\tilde{T}_i),..., \theta(T_l)) = (0, n^{rk-2}, n-\delta, n+\delta)\]
where $\delta := \theta(U) - x$. By assumption this must be similar to $(0,n^{rk})$ 
which is only possible if they are actually equal (here we used that $l \ge 4$).
Therefore $\delta = 0$ and $x=\theta(U) $.
As this holds for any  subsequence $U' \subseteq T_1$ of  the same length $d < |T_1|$, 
necessarily $T_1=(f^n)$ for some  generator  $f \in Z_{nr}$ such that $f + Z_r =e$. 
Moreover, as $x =\theta(U) =  df$, we get by the definition of $d$ and $||x||_f$ that
\begin{align} \label{eq:autista}
||x||_f = || x + Z_r||_e
\end{align}
for every $x \in T_i$, where $i$ differs from that unique index $s$ for which $\theta(T_s) = 0$. 
Observe on the other hand that \eqref{eq:autista} cannot be true for every element $y \in T_s$,
for otherwise $|| T_s||_f = ||T_s/Z_r||_e = n$, which is impossible, as $||T_s||_f$  must be a multiple of $nr$. 
Now suppose that $|T_s| \ge 2$ and that \eqref{eq:autista} fails for $y \in T_s$. 
Then swapping $y$ with a proper subsequence  $U \subseteq T_1$ of length $ ||y+Z_r||_e$
we get as before that $\delta:= \theta(U) - y = -nf$, whence $||y||_f = ||y+Z_r||_e +n(r-1)$.
On the other hand if $z \in T_s$ is a second element besides $y$ for which \eqref{eq:autista} fails, 
then in particular $(yz) \neq T_s$, as otherwise calculating $||z||_f$ by the same argument yields that
 $||T_s||_f = ||T_s/Z_r||_e + 2n(r-1) = n(2r-1)$, which is not  a multiple of $nr$. 
 Now swapping $(yz)$ with a proper subsequence of $T_1$ of length $||yz +Z_r||_e$ gives
 a $Z_r$-contraction of $S$ of the form $(2n,-n,n^{rk-2})$ which is not similar to $(0,n^{rk})$.
 This contradiction shows that $y$ is unique with the property that $||y||_f \neq ||y + Z_r||_e$.
 So if $|T_s| \ge 3$ then the sequence $S'$ obtained from $S$ by replacing $T_s$ with $T_s^b$
 will not satisfy this requirement:  indeed, $||x+Z_r||_e=||x^b+Z_r||_e$ for all $x$, whereas $||x||_f=||x^b||_f$ means $x\in Z_n\subset Z_{nr}$.  Thus  $S'$ will have $Z_r$-contractions not similar to $(0,n^{rk})$,
 which is a contradiction as $S' \sim S$. 
 It remains that $|T_s| = 2$ and $T_s = (-y,y)$. Then necessarily  $s \in \{  l-1,l \}$ and $T_1=\dots=T_{l-2}=(f^n)$. 
 If moreover $y \neq -f$ then $n(r-1) < ||y||_f < nr -1$ and consequently we have the factorization
 $T_sT_1 = (-y,y, f^n) = (y, f^{nr - ||y||_f})(-y, f^{||y||_f - n(r-1)})$
 which leads us back to the case when $|T_s| \ge 3$. 
 Finally, if $y = -f$ then observe that $f^b \neq \pm f$, as we have $n >2$;
hence after replacing $T_s$ with $T_s^b \neq (-f,f)$ we get back to the case when $y \neq -f$. 

 As a result of these contradictions we excluded that $|T_s| \ge 2$. 
 Therefore $|T_s| =1$ and $T_s=(0)$. Then
 we must have $|T_i| = n$ for every $i \neq s$ whence $|T_i/ Z_r| = (e^n)$ follows. Using \eqref{eq:autista}
this implies that  $S=(0,f^{nrk})$.
\end{proof}

\begin{theorem} \label{thm:G_1} 
For the group $G= Z_s \times (Z_r \rtimes_{-1} Z_{2^{n+1}})$, where  $r \ge 3$, $n \ge 1$ and $r,s$ are coprime odd integers,
we have $\beta_k(G) = 2^nsr k + 1$, except if $s=n=1$, in which  case $\beta_k(G) = 2rk+2$.
\end{theorem}

\begin{proof} 
$\beta_k(G)$ is the length of a sequence $S$ over $A:= Z_{2^nsr}$ which is $k$-extremal with respect to $\tau_A^G$. 
By Proposition~\ref{prop:kontra} any $Z_{r}$-contraction of any sequence equivalent to $S$ must be $k$-extremal with respect to $\tau_{Z_{r}}^{D_{2r}}$,
hence it is similar to $(0,(2^ns)^{rk})$ by Proposition~\ref{extrem_Dn}. Therefore $S$  is similar to $(0,1^{2^nsrk})$ by Lemma~\ref{lemma:kontra},
provided that $2^n s\ge 3$; in particular $\beta_k(G) = |S| =  2^nsr k + 1$. 

For the case $s=n=1$ we have $\beta_k(Z_r\rtimes_{-1}Z_4)\le 2\beta_{k}(D_{2r})=2r+2 $
by Lemma~\ref{lemma:reduction} and Corollary~\ref{cor:dihedralbeta}.
To see the reverse inequality consider the representation on $V=\F^2$ of $G:=\bra a,b\ket $ given by the matrices 
\begin{align}\label{Qrep}
 a\mapsto \left(\begin{array}{cc}\omega & 0 \\0 & \omega^{-1}\end{array}\right) \qquad
b\mapsto  \left(\begin{array}{cc}0 & i \\i & 0\end{array}\right)
\end{align}
where $\omega$ is a primitive $2r$-th root of unity  and $i = \sqrt{-1}$  a primitive fourth root of unity.
Then $\F[V] =\F[x,y]$ where $x,y$ are the usual coordinate functions on $\F^2$.  Obviously $(xy)^2$ is invariant under $a$ and $b$ alike; from this it is easily seen that $R=\F[V]^{G}$ is generated by 
$(xy)^2$, $\tau^G_A(x^{2r})$ and $\tau^G_A(x^{2r+1}y)$. This shows that any element of $R_+^{k+1}$ not divisible by $(xy)^2$
must have degree at least $2r(k+1)$. As a result $(R_+^{k+1})_{2rk + 2} \subseteq \bra (xy)^2 \ket$. 
The invariant $\tau^G_A(x^{2rk+1} y) \in R_+$ of degree $2rk + 2$  does not belong to the ideal $\bra (xy)^2 \ket$
and this proves  that $\beta_k(G) \ge 2rk + 2$.
\end{proof}


\section{The quaternion group}\label{sec:quaternion} 

The dicyclic group $Dic_{4n}$ is defined for any $n>1$ by the presentation
\[ Dic_{4n} = \bra a,b: a^{2n}=1, b^2 = a^n, bab^{-1} = a^{-1} \ket \]
In particular for $n=2$ we retrieve the quaternion group $Q = Dic_8$. 
The equality $\beta(Q)=6$ for $\F=\mathbb{C}$ was proved in \cite{schmid}.

\begin{proposition}\label{prop:betadiciklikus}  
We have $\beta_k(Dic_{4n})=2nk+2$ for  $n>1$ even  and $k\ge 1$. 
Moreover if $(r,4n) =1$  then $1 \le \beta_k(Z_r\times Dic_{4n}) - 2nrk \le 2$.
\end{proposition}

\begin{proof} Taking $\omega$ a primitive $2n$-th root of unity  in \eqref{Qrep}, the same argument as in 
 the proof of Theorem~\ref{thm:G_1} shows that  $\beta_k(Dic_{4n}) \ge 2nk + 2$.
Moreover for $G:=Z_r\times Dic_{4n}$ we have $\beta_k(G)\ge 2rnk +1$ by Theorem~\ref{thm:lower}. 
Observe that $G/Z(Dic_{4n})$ is isomorphic to $Z_r\times D_{2n}$, respectively to $Z_{2r} \times Z_2$ for $n=2$. 
Combining Lemma~\ref{lemma:reduction}  with Corollary~\ref{cor:dihedralbeta} 
leads to the inequality $\beta_k(G) \le  2nrk+2$.
\end{proof}

\begin{proposition} \label{Q extrem}
Let  $Q = \bra a,b\ket$ be the quaternion group where $A:=\bra a \ket$ is isomorphic to $  Z_4$. 
If $S$ is a zero-sum sequence over $\hat A$ of length $4k+2$ which is $k$-extremal with respect to $\tau^Q_A$ 
then  $S=(1^{t}, 3^{s})$ where  $t \neq s$. 
\end{proposition}

\begin{proof} Set $I = \F[V]^{A}$, $R= \F[V]^Q$ and
let $S = (0^x, 2^y, 1^t,3^s)$ be the weight sequence of a monomial $m \in I$ of degree $4k+2$ such that $\tau(m) \not\in R_+^{k+1}$. 
By replacing $m$ with $m^b$, if needed, we may suppose that $t \ge s$. 
We will use induction on $t-s$.
Suppose first that $t-s \le 4$ and consider the factorization $S=(22)^{\lfloor y/2 \rfloor} (13)^s (0)^x T$.
If $y$ is odd then necessarily $T=(211)$ and $x \ge 1$,
hence $m \in I_+^{2k+1}$, which is a contradiction by Lemma~\ref{lemma:reduction}.
If however $y$ is even then either $T$ is empty, and then $m \in I_+^{2k+1}$ again, 
or else $T=(1111)$; in this later case if $x \ge 2$ then again $m \in I_+^{2k+1}$ or
otherwise, taking into account that $|S|$ is even, it remains that $x=0$ and $S=(2^y, 1^{s+4} ,3^s)$.
Now, if $y>0$ then take a factorization $m=uv$ such that $\Phi(u) = (211)$ and observe that
$\tau(m) = \tau(u)\tau(v) - \tau(u^bv) \in R_+^{k+1}$, because on the one hand $\deg(v) = 4k-1 > \beta_{k-1}(Q)$,
while on the other hand $\Phi(u^bv) =(2^y)(13)^{s+2}$, hence $u^bv \in I_+^{2k+1}$ by what has been said before. 
From this contradiction we conclude that $y=0$ and  $S=(1^{s+4},3^s)$
whenever $t-s \le 4$ holds. 

Finally, if $t-s> 4$ we have a factorization $m=uv$ with $\Phi(u)=(1111)$,
and since $\tau(m) = \tau(u)\tau(v) - \tau(u^bv) \not\in R_+^{k+1}$ by assumption,
it is necessary that either $\tau(v) \not \in R_+^k$, when   $\Phi(v) = (1^{t-4}, 3^s)$  by the induction hypothesis, 
or $\tau(u^bv) \not\in R_+^k$, when similarly $\Phi(u^bv) = (1^{t-4}, 3^{s+4})$, 
and in both cases $\Phi(m) = (1^t, 3^s)$, as claimed. 
\end{proof}

\begin{theorem} \label{thm: beta Zp Q}
Let $G =Z_p \times Q$ for an odd prime $p$. 
Then $\beta_k(G) = 4p k +1$ for every $k \ge 1$.
\end{theorem}

\begin{proof}
Here the distinguished abelian normal subgroup 
is $A:=C\times B\cong Z_{4p}$, where $C:=Z_p\triangleleft G$ and $B:=\bra a\ket$. Set $L:=\F[V]$ and $R:=L^G$. 
We write $\theta|_C$ and $\theta|_B$ for the restriction of the character $\theta \in \hat{A}$ to $C$ or $B$, respectively, 
and we define accordingly $S|_C$ and $S|_B$ for any sequence $S$ over $\hat{A}$; 
note that $\theta = (\theta|_C, \theta|_B)$ by the natural isomorphism $\hat A\cong \hat C\times \hat B$.

We already proved in Proposition~\ref{prop:betadiciklikus} that $1 \le \beta_k(G) -4kp \le 2$. 
Suppose for contradiction that there is a $G$-module $V$ and a  monomial $m\in \F[V]^A$ with $\deg(m)=4pk+2$ and $\tau^G_A(m)\notin R_+^{k+1}$.
Given that the restriction of $\tau^Q_B$ to $L^A$ coincides with $\tau^G_A$, the sequence $\Phi(m)|_B$ is $kp$-extremal: indeed, otherwise 
$\tau^Q_B(m)\in (L_+^Q)^{kp+1}$ as $\deg(m)=\beta_{kp}(Q)$, and since $(L_+^Q)^{kp+1}\subseteq R_+^kL_+^Q$ by Lemma~\ref{lemma:reduction}, we get that $\tau^G_A(m)=\tau^Q_B(m)\in (R_+^kL_+^Q)\cap R_+$, 
but for any $f\in (R_+^kL_+^Q)\cap R_+$ we have $f=\frac1{[G:B]}\tau^G_A(\tau^A_B(f))\in \tau^G_A(R_+^k\tau^A_B(L_+^Q))\subseteq R_+^{k+1}$,
a contradiction.  
As a result  $\Phi(m)|_B=(1^t,3^s)$ by Proposition~\ref{Q extrem}, where $t > s$ can be assumed and $t+s=4pk+2$. 
Accordingly $m$ has a factorization 
\begin{align}\label{eq:Q fact} 
m=m_1\cdots m_l
\end{align} 
where $\Phi(m_i)|_B=(1,3)$ for $i\le s$ and $\Phi(m_i)|_B=(1^4)$ for $s < i\le l$,
so that $l=s+\frac{t-s}4$.  
Consider the sequence $S := (\theta(m_1)|_C,...,\theta(m_l)|_C)$; 
it contains at most one occurrence of $0$, 
for otherwise $m\in (L_+^A)^2L^A_{\ge 4pk-6}\subseteq R_+L^A_{>\beta_{k-1}(G)}$ by Lemma~\ref{lemma:reduction}, 
hence $\tau^G_A(m)\in R_+^{k+1}$, a contradiction. 
Moreover $S$ cannot be factored into $2k+1$ zero-sum sequences over $\hat C$, 
for otherwise  $\tau^G_A(m)\in R_+^{k+1}$ follows again as  
$m\in (L_+^A)^{2k+1}\in R_+^kL_+^A$.

We claim  that $\{ 1,...,l\}$ can be partitioned into two disjoint, non-empty subsets $U,V$ such that 
 the monomials $u = \prod_{i \in U} m_i$ and $v = \prod_{i \in V} m_i$ 
are $A$-invariant, $\tau_A^G(u^bv) \in R_+^{k+1}$ and $\deg(v) > 4p(k-1)+2 \ge \beta_{k-1}(G)$. 
Under these assumptions 
$\tau^G_A(m)=\tau^G_A(u)\tau^G_A(v)-\tau^G_A(u^bv)\in R_+^{k+1}$, since $\tau_A^G(v) \in R_+^k$
 and this will   refute our indirect hypothesis.

We will prove our claim by induction on $\frac{t-s}{4}$. 
Suppose first that $\frac{t-s}{4}=1$, i.e. $l=2pk$. Then $\theta(m_1)|_C=...=\theta(m_l)|_C$ for otherwise $S$ could be factored into $2k+1$ zero-sum sequences. 
Observe that if $x$ is a variable in $m_i$ and $y$ is a variable in $m_j$ where $i\neq j$ and $\theta(x)|_B=\theta(y)|_B$, 
then $\theta(x)=\theta(y)$, since otherwise swapping the variables 
$x$ and $y$ yields another factorization as in \eqref{eq:Q fact} where $l=2pk$ but not all $\theta(m_i)|_C$ are equal. 
We conclude that $\Phi(m) = (e^{2pk+3}, (3e)^{2pk-1})$ for some generator $e$ of $\hat{A}$. 
Then $U:=\{1,...,p\}$, $V := \{p+1,...,l\}$ is the required bipartition,
since $\Phi(u^bv)  $ is not similar to $\Phi(m)$ and consequently $\tau_A^G(u^bv) \in R_+^{k+1}$ by the above considerations.

For the rest it remains that $\frac{t-s}4 > 1$, hence $\Phi(m_{l-1})|_B=\Phi(m_l)|_B=(1^4)$. 
If $\theta(m_i)=0$ for some $i > s$, say $i=l$, then choosing $U= \{l\}$ gives the required factorization: 
indeed, $\Phi(u^bv)|_B=(1^r,3^s)$ where $r-s < t-s$ and consequently 
$\tau_A^G(u^bv) \in R_+^{k+1}$   by induction on $\frac{t-s}{4}$. 
If however $S$ contains at least $p+1$ non-zero elements then using Lemma~\ref{lemma:easy} 
we get a subset $I \subset \{1,...,l-2 \}$ such that $|I| \le p-1$ and $\theta(\prod_{i \in I} m_i) = - \theta(m_l)$.
Now set $U := I \cup \{ l\}$, $V := \{ 1,...,l-1\} \setminus I$ and observe that
$\Phi(u^bv)|_B=(1^r,3^s)$ where $r-s < t-s$. So we are done as before, 
provided that $|U| \le p-1$ or there is an index $i \in U$ such that $i \le s$,
because this guarantees that $\deg(u) \le 4p-2$. 

Otherwise it remains that 
$l = p+1$, $s=1$ and $\theta(m_1) =0$.
Here $m_1=xy$, where $\theta(x)|_B=1$ and $\theta(y)|_B=3$. 
If there is a variable $z$ in $m_2\dots m_l$ with $\w_C(z)\neq \theta_C(x)$, 
then by swapping the variables $x$ and $z$ we get back to a case considered already. 
Thus  $\Phi(m/y) = ((1,c)^{4p+1})$ for a non-zero element $c\in \hat C\cong Z_p$, 
and $\theta(y)=(3,-c)$. Here  $U:= \{1\}$, $V:=\{2,...,l\}$ is the required bipartition, 
because $\Phi(u^bv)=((1,-c),(3,c),(1,c)^{4pk})$, and since $c\neq -c$ it follows by the above considerations that $\tau^G_A(u^bv)\in R_+^{k+1}$. 
\end{proof}


\section{Proof of the main result}  \label{sec:index2} 

We shall use  for the semidirect product of two cyclic groups the notation: 
\[ Z_m \rtimes_d Z_n = \bra a,b:   a^m=1, b^n=1, bab^{-1}=a^d \ket  \quad \text{ where } d \in \mathbb{N}\mbox{ is coprime to }m \]

\begin{proposition}[Burnside 1894,  
 see for example \cite{brown} ch. IV.4] \label{burn} 
If $G$ is a finite $p$-group with a cyclic subgroup of index $p$ then it is one of the following:
\begin{enumerate}
\item $Z_{p^n}$ \quad ($n\geq 1$)
\item $Z_{p^{n-1}}\times Z_p$ \quad 
($n\geq 2$)
\item $M_{p^n} := Z_{p^{n-1}} \rtimes_d  Z_p \qquad d={p^{n-2}+1}$ \quad
($n\geq 3$)
\item $D_{2^n} 	:= Z_{2^{n-1}} \rtimes_{-1} Z_2$
\quad  ($n\geq 4$) 
\item $SD_{2^n} := Z_{2^{n-1}} \rtimes_d Z_2 \qquad d={2^{n-2}-1}$ 
\quad ($n\geq 4$) 
\item $Dic_{2^n} := \langle a,b\mid a^{2^{n-1}}=1, b^2=a^{2^{n-2}}, bab^{-1}= a^{-1}\rangle$ \quad 
($n\geq 3$) 
\end{enumerate} 
\end{proposition}

Let $H$ be one of the $2$-groups in the above list, 
$\bra a\ket$ an index $2$ subgroup in $H$, and $b\in H\setminus \bra a \ket$, so that $H=\bra a,b\ket$.
If $H$ is a $2$-group as in case (3)--(6) of Proposition~\ref{burn} then
for any odd integer $r > 1$ it is customary to denote by $M_{r2^n}$, $D_{r2^n}$, $SD_{r2^n}$, $Dic_{r2^n}$
the group $Z_r \rtimes_{-1} H$, where $b \in H$ acts on $Z_r$ by inversion $x \mapsto x^{-1}$
and $\bra a\ket$ centralizes $Z_r$.

\begin{proposition}\label{prop:index2cyclic}
Any finite group containing a cyclic subgroup of index two is isomorphic to
\[Z_s\times (Z_r\rtimes_{-1} H)\]
where $r,s$ are coprime odd integers, and $H$ is a $2$-group in Proposition~\ref{burn}. 
\end{proposition} 

\begin{proof} Let $G$ be a finite group with an index two cyclic subgroup $C$. 
Then $C$ uniquely decomposes  as $C=Z_m\times Z_{2^{n-1}}$ for some odd integer $m>0$ and $n\geq 1$. As $Z_m$ is a characteristic subgroup of $C$, it is normal in $G$. 
Thus by the Schur-Zassenhaus theorem $G=Z_m\rtimes H$ for a Sylow $2$-subgroup $H$ of $G$. 
Moreover, the characteristic direct factor $Z_{2^{n-1}}$ is also normal in $G$, hence 
we may suppose that it is identical to the index two cyclic subgroup $\bra a \ket \le H$ 
(as the automorphism group of $H$ acts transitively on the set of index two subgroups of $H$). 
Now $Z_m$  decomposes uniquely as a direct product $Z_m=P_1\times\cdots\times P_l$ of its Sylow subgroups.  
After a possible renumbering we may assume that $H$ centralizes $P_1,\ldots,P_t$, 
and $H/\bra a \ket$ acts on $P_{t+1},\ldots,P_l$ via the automorphism $x\mapsto x^{-1}$. 
Setting $Z_s:=P_1\times\cdots\times P_t$, $Z_r:=P_{t+1}\times\cdots\times P_l$ 
we obtain the desired conclusion.  
 \end{proof}


\begin{theorem}\label{thm:betaindex2} 
If $G$ is a non-cyclic group with a cyclic subgroup of index two then
\[ \beta_k(G) = \frac{1}{2} |G| k +
\begin{cases}
2 	& \text{ if } G=Dic_{4n}, \text{  $n$ even}\\
	& \text{ or } G = Z_r \rtimes_{-1} Z_4, \text{ $r$ odd } \\
1	& \text{ otherwise }
\end{cases}\]
\end{theorem} 

\begin{proof}
If $G$ is any group with a cyclic subgroup $A= \bra a \ket$ of index $2$, then 
Theorem~\ref{thm:lower} gives us the following lower bound:
\[ \beta_k(G) \ge \beta_k(A) + \D(G/A) -1 = k|A| + \D(Z_2) -1 = \tfrac{1}{2}|G| +1 \]
To establish the precise value of the generalized Noether number $\beta_k$ for these groups,
by Proposition~\ref{prop:index2cyclic} we will have to  consider  the groups of the form $G:= Z_s \times (Z_r \rtimes_{-1} H)$
where $H$ is one of the groups of order $2^n$ listed in Proposition~\ref{burn}. 
In all these cases  $\beta_k(G) \le \beta_{sk}(Z_r \rtimes_{-1} H)$  by Lemma~\ref{lemma:reduction}.

(1) If $H =Z_{2^n}$ then by Theorem~\ref{thm:G_1} we have $\beta_k(G)  = 2^{n-1} rs k + 1$ except if 
$n=2$ and $s=1$, in which case $\beta_k(G)  = 2^{n-1} rs k + 2$

(2) If $H =Z_2 \times Z_{2^{n-1}}$ by the isomorphism $Z_r \rtimes_{-1} (Z_{2} \times Z_{2^{n-1}}) \cong Z_{2^{n-1}} \times D_{2r}$ 
we get from the application of Lemma~\ref{lemma:reduction} and Corollary~\ref{cor:dihedralbeta} that
\begin{align}
\beta_k(G) \le \beta_{sk}(Z_{2^{n-1}} \times D_{2r}) \le \beta_{2^{n-1}sk}(D_{2r}) \le 2^{n-1} rs k  +1 
\end{align}

(3) If $H= M_{2^n}$ then the group $Z_r \rtimes_{-1} M_{2^n} =  M_{2^nr}$ will contain a subgroup $C = \bra a^2,b \ket \cong Z_{2^{n-2}} \times D_{2r}$.
The  subgroup $N:=Z_s\times C$ has index $2$ in $G$ and falls under case (2), hence by Lemma~\ref{lemma:reduction} and case (2) we have
\begin{align} 
\beta_k(G)=\beta_{2k}(N)=2^{n-1}krs+1
\end{align}

(4) If $H= D_{2^n}$ then $G = Z_s \times D_{2^nr}$ and we are  done by Corollary~\ref{cor:dihedralbeta}

(5) If $H = SD_{2^n}$ then the group $Z_r \rtimes_{-1} SD_{2^n} = SD_{2^nr}$  contains a subgroup $B= \langle a^2,b\rangle\cong D_{2^{n-1} r}$. 
Observe that  $B$ is a normal subgroup, as it has index $2$,
hence by Lemma~\ref{lemma:reduction} and Corollary~\ref{cor:dihedralbeta} we get that
\begin{align} 
\beta_k(G) \le \beta_{sk}(SD_{2^nr}) \le \beta_{2sk} (D_{2^{n-1} r}) \le 2^{n-1}rsk +1
\end{align}

(6) If $H = Dic_{2^n}$ then for $n=2$ we get back to case (2), as $Dic_4 = Z_2 \times Z_2$;
if however $n \ge 3$ then the quaternion group $Q$ is a subgroup of index $2^{n-3}r$ in $Z_r \rtimes_{-1} H$, 
therefore by Proposition~\ref{prop:betadiciklikus} we have $\beta(G) = 2^nrs k +2$ if $s=1$ and
for $s > 1$ we get using Lemma~\ref{lemma:reduction} combined with Theorem~\ref{thm: beta Zp Q} 
that for any prime $p$  dividing $s$:
\begin{align} 
\beta_k(G) \le \beta_{k2^{n-3}r} (Z_s \times Q) \le \beta_{k2^{n-3}rs/p} (Z_p \times Q) \le 2^{n-1}rsk +  1
\end{align}
With this all possibilities are accounted for and our claim is established.
\end{proof}


\end{document}